\documentclass{article}
\usepackage[utf8]{inputenc}
\usepackage{blindtext}
\usepackage{amssymb}
\usepackage{amsthm}
\usepackage{amsmath}
\usepackage{hyperref}
\numberwithin{equation}{section}

\newtheorem{theorem}{Theorem}[section]
\newtheorem{lemma}{Lemma}[section]
\newtheorem{definition}{Definition}[section]

\newtheorem{proposition}{Proposition}[section]

\newtheorem{remark}{Remark}[section]

\title{Elements of Convex Geometry in Hadamard Manifolds with Application to Equilibrium Problems}
\author{
G. C. Bento\thanks{The author was partially supported by FAPEG 201710267000532, PRONEX/FAPERJ and CNPq Grants 310864/2017-8, 314106/2020-0. IME-Universidade Federal de Goi\'as,
Goi\^ania-GO 74001-970, BR (Email: {\tt glaydston@ufg.br}) - {\bf Corresponding author}}
\and
J. X. Cruz Neto
\thanks{The author was partially supported by CNPq Grants 308330/2018-8. CCN, DM-Universidade Federal do Piauí,
Teresina, PI 64049-550, BR (Email: {\tt jxavier@ufpi.edu.br})  }\and
I. D. L. Melo  \thanks{This author was partially supported by FAPEPI.  CCN, DM-Universidade Federal do Piauí,
Teresina, PI 64049-550, BR (Email: {\tt italodowell@ufpi.edu.br }).}
}
\date{}

\begin{document}

\maketitle

\begin{abstract}
In this paper, is introduced a new proposal of resolvent for equilibrium problems in terms of the Busemann's function. A great advantage of this new proposal is that, in addition to be a natural extension of the proposal in the linear setting by Combettes and Hirstoaga in \cite{combettes2005equilibrium}, the new term that performs  regularization is  a convex function in general Hadamard manifolds, being a first step to fully answer to the problem posed by Cruz Neto et al.in \cite[Section 5]{cruz2017note}. 
During our study, some elements of convex analysis are explored in the context of Hadamard manifolds, which are interesting on their own. In particular, we introduce a new definition of  convex combination (now commutative) of any finite collection of points and present the realization of an associated Jensen-type inequality.  
\end{abstract}
{\bf Keywords}: Equilibrium problem; KKM's lemma; Helly's theorem; Jensen's inequality; Hadamard manifold\\
{\bf Subclass}: 47N10; 47H05; 52A37


\maketitle

\section{Introduction}
In this paper, some elements of convex analysis are explored in the context of Hadamard manifolds, among which we highlight the KKM's lemma. It is was introduced in 1929 by three Polish mathematicians, Knaste, Kuratowski, Mazurkiewicz; see \cite{Knaster1929}, and works as follows: given  \mbox{co}$(\{x_1,\ldots, x_{n+1}\})$ (the convex  hull  of $\{x_1,\ldots,x_{n+1}\}$) and $C_1, \ldots, C_{n+1}$ closed subsets of $\mathbb{R}^n$, if for each subset $I\subset \{1,\ldots, n+1\}$ one has \mbox{co}$(\{x_i:i\in I\})\subset \cup_{i\in I}C_i$, then $\cap_{i=1}^{n+1}C_i\neq\emptyset$. There is a vast literature dealing with generalizations of this simpler version of KKM lemma. In the linear context, \cite{granas2003} is an important 
reference with excellent discussion and main references on the topic as well as important applications, among which we mention the existence of solution for equilibrium problems. On this specific point, see, for example, \cite{KyFan1992} and the references therein. When highlighting equilibrium problems, other problems such as optimization problems, Nash equilibrium problems, complementarity problems, fixed point problems and variational inequality problems are also considered, since all these problems can be formulated as equilibrium problems; see, for example, \cite{blum1994,bianchi1996} and the references therein. Regardless of this  in the linear setting, we would like to mention about reference \cite{de2019}, which does not only present a beautiful background on KKM's lemma with a rich list of references on the subject, but also presents important connections with the theorem of Helly and Carathéodory. In the Riemannian context, we highlight the following references that deal the KKM's lemma \cite{colao2012equilibrium,zhou2019,park2019riemannian}. As a special case, see \cite{niculescu2009fan,park2020coupled}, in which an approach in other contexts that has Hadamard manifold as a particular case have been presented. The main purpose of generalizing the KKM's lemma to a Hadamard manifold is to establish  existence of solution and, in particular, to ensure the well-definedness of the resolvent and the proximal point method for equilibrium problems. A definition of resolvent in the Riemannian context associated with a bifunction $F(\cdot,\cdot)$ has been presented in \cite{colao2012equilibrium}. It is a set-valued operator $J_\lambda^F:M\rightrightarrows\Omega$, $\lambda>0$, which is given as follows:
\begin{equation}\label{bi-Function}
J_{\lambda}^F(x):=\{z\in \Omega : \lambda F(z,y)-\langle \mbox{exp}^{-1}_zx,\mbox{exp}^{-1}_zy\rangle\geq 0,\quad y\in \Omega\}. 
\end{equation}
Despite (\ref{bi-Function}) being a natural extension
of the resolvent introduced in the linear setting in \cite{combettes2005equilibrium}, the well-definedness of the resolvent and consequently of the proximal point algorithm for solving equilibrium problems depends on the convexity of the function $M\ni y\mapsto \langle u_z, \mbox{exp}^{-1}_zy\rangle$, $u_z\in T_zM$, that has been shown not to happens in general; see \cite{wang2016some,cruz2017note}. An important contribution of this paper is the introduction of a new proposal of a resolvent in terms of the Busemann's function. The great advantage of this new proposal is that in addition to being a natural extension of that presented in \cite{combettes2005equilibrium}, the new term that performs  regularization is  a convex function in general Hadamard manifolds being a first step to fully answer to the problem posed in \cite[Section 5]{cruz2017note}.

Another important convex analysis result that we explore in this present paper is the Helly's theorem. It was introduced by Edward Helly in the linear setting in \cite{helly1923} and gives sufficient conditions for a family of convex sets to have a nonempty intersection. Over the years, a large variety of proofs as well as applications have been presented; see, for example, \cite{eggleston1958, danzer1963helly, lay2007}, where the relation with other important classic results of convex geometry can also be found.
As far as we know, the first approach to Helly's theorem in the Riemannian context was presented in \cite{ledyaev2006helly}, where specifically the authors generalized the classical Helly’s theorem concerning the intersection of convex sets in $\mathbb{R}^n$ for manifolds of nonpositive curvature (for example, Hadamard manifolds). The main result presented in \cite{ledyaev2006helly} considers a certain $(CC)-$condition on a subset $K$ of the Riemannian manifold, which ensures that the convex hull of any finite set of points  $D\subset K$ is a compact set. In the linear setting, this condition is obtained as a consequence of the Carathéodory's theorem, one of the pillars of combinatorial convexity introduced by Constantin Carathéodory in \cite{caratheodory1907variabilitatsbereich}. However, as noted in \cite[Remark~3.2]{zhou2019}, its validity is not known even in Hadamard manifolds.
The absence of a Carathéodory theorem leads to some significant obstacles, for example, ``what is the convex hull of three points in a 3 or higher dimensional Riemannian manifold?" This issue was highlighted by Berger in \cite[page 253]{berger2012panoramic}, who also conjectured that said convex hull is not closed, except in very special cases. Based on the aforementioned discussion, as an additional contribution of our paper we mention  provision of an alternative proof for \cite[Theorem 4.2]{ledyaev2006helly} (see Theorem~\ref {Hellymain1} in the present paper), which is new even in a linear setting. In addition, we also present a proof for a version of the KKM's lemma using the Helly's theorem (see Theorem~\ref{TheoremKKM}). It is worth mentioning that others versions of the KKM lemma in the Riemannian setting can be found, for example, in \cite{colao2012equilibrium,zhou2019,park2019riemannian}. However, since this topic dealing with combinatorial convexity shows to have an interdisciplinary character, the version of the KKM lemma presented in the last result seems more appropriate, inclusively, for our purposes dealing with existence result of solutions for equilibrium problems.

As a result of the contributions  aforementioned, some new results of convex analysis in Hadamard manifolds are also introduced, which are interesting on their own. We highlight the following:
\begin{itemize}
\item[a)] Convexity of the interior and of the closure of a convex set; see Proposition~\ref{lemmaBBF2020};

\item[b)] Upper semicontinuity of the Busemann's function in the variable that determines the ray from which the Busemann function is defined; see Lemma~\ref{lemmaUpperSem}.

\end{itemize}
In addition, taking into account the notion of the convex hull explored in the discussion involving Helly's theorem and KKM's lemma, and based on the notion of pseudo-convex combination presented in \cite{zhou2019}, we introduce a new definition of convex combination (now commutative) of any finite collection of points, and we prove Jensen's inequality by considering both the pseudo-convexity and our proposal of convex combination. 
 It is worth mentioning that this important inequality is attributed to the Danish mathematician Johan Jensen, \cite{jensen1906}. It has appeared in the nonlinear setting where the ``convex combination" of the points involved is the ``center of mass" (barycenter) of the points both in the discrete case and also in the continuous case associated with a measure of probability; see, for example, \cite{jost1994,sturm2003probability, Bacak2014:1} and their references therein.

The remainder of this paper is organized as follows: in Section~\ref{section2}, some notations, terminology,  concept and results related to Riemannian geometry are presented. We also record and introduce some basic concepts and results of convex analysis. In Section~\ref{section3}, we explore the Helly’s theorem from a theoretical viewpoint and obtain an alternative version for the Knaster-Kuratowski-Mazurkiewicz theorem, also known as KKM's lemma. In Section~\ref{section4}, we explore KKM's lemma to establish existence of solutions for equilibrium problems and well-definedness of a new resolvent associeted with equilibrium problems.

\section{Preliminaries}\label{section2}
In this section, we set the notations, terminology and some pertinent concept and results related to Riemannian geometry. We also introduce some basic concepts and results of convex analysis, which are interesting on their own, including due to the original proofs here built. 

\subsection{Notation and terminology}
 this section, we present some pertinent concept and results related to Riemannian geometry. For more details see, for example, \cite{carmo1992,sakai1996,paternain2012geodesic}.

Assuming that $M$ is a complete and connected Riemannian manifold,  from Hopf-Rinow theorem it is known that any pair of points in $M$ can be joined by a minimal geodesic. Moreover, $(M,d)$ is a complete metric space, where $d$ denotes the Riemannian distance, and bounded closed subsets are compact. We denote by $T_xM$ the  tangent space of $M$ at $x$, by $TM=\cup_{x\in M}T_xM$ the tangent bundle of $M$.
Let $g$ be the Riemannian metric of $M$, also denoted by $\langle  ~\cdot, \cdot  ~ \rangle$, with the corresponding norm given by $\| \cdot \|$ and $\pi\colon TM\to M$ the canonical projection.  For $\theta=(x,v)\in TM$,  let $\gamma_{\theta}(\cdot)$ denote the unique geodesic with initial conditions $\gamma_{\theta}(0)=x$ and $\gamma_{\theta}'(0)=v$. For a given $t\in \mathbb{R}$, let $\phi^t:TM \to TM$ be the diffeomorphism given by  $\phi^t(\theta)=(\gamma_{\theta}(t),\gamma_{\theta}'(t))$. Recall that this family is a flow (called the geodesic flow) in the sense that  $\phi^{t+s}(\cdot)=(\phi^{t}\circ \phi^{s})(\cdot)$ for all $t,s\in \mathbb{R}$. 
The exponential map $\mbox{exp}:TM\to M$ is defined by $\mbox{exp}(\theta):=\gamma_{\theta}(1)$. For $x\in M$ fixed, $\mbox{exp}_xv :=\gamma_{v}(1,x)$. 
Consider $TM$ furnished with the usual Sasaki metric, the projection $\pi:TM \to M$ is continuous. For $x,y\in M$, denotes by $\Gamma_{x,y}$ the collection of all $C^1$ curves joining $x,y$. The function  $M\ni (x,y)\mapsto d(x,y)=\mbox{inf}\{l(\gamma):\gamma\in \Gamma(x,y)\}$, where
\[
l(\gamma):=\int_{0}^1\|\gamma'(t)\|dt \quad (\mbox{length of}\; \gamma),
\]
represent the  Riemannian distance.
Given a point $x\in M$ and $D\subset M$, the distance from $x$ to $D$ is defined by $d_D(x):=\mbox{inf}\{d(y,x): y\in D\}$.
A complete, simply connected Riemannian manifold of nonpositive sectional curvature is called a Hadamard manifold. The following result is well known (see, for example, \cite[Theorem 4.1, p. 221]{sakai1996}).

\begin{proposition} 
Let $M$ be a Hadamard manifold and $x\in M$. Then, $\emph{exp}_{x}(\cdot)$ is a diffeomorphism, and for any two points $x,y\in M$ there exists an unique normalized geodesic joining $x$ to $y$, which is, in fact, a minimal geodesic.  Note that $d(x,y)=\|\emph{exp}_{x}^{-1}y\|.$
\end{proposition}
Throughout the remainder of this paper, unless otherwise stated, we always assume that $M$ is an $n-$dimensional Hadamard manifold.

\subsection{Convex Analysis Elements}\label{Sec2.2}
In this section, we record some basic concepts and results of convex analysis. From the concept of pseudo-convex combination presented in \cite{zhou2019}, we introduce a new notion of convex combination, now commutative, we extend the result involving a Jensen-type inequality to these more general convex combinations, we highlight and analyze an important conjecture attributed to Berger that is associated with the convex hull of a finite set of points in the Riemannian context. In addition, some basic results of convex analysis, not identified in the literature and which are useful in the rest of the paper, are presented and proved.

Since $M$ was assumed to be a Hadamard manifold, for any $x,y\in M$, there exists an unique minimal geodesic $\gamma_{x,y}(t):=\mbox{exp}_{x}(t\mbox{exp}^{-1}_xy)$, $t\in [0,1]$, joining $x$ to $y$. 
A set $D\subseteq M$ is said to be convex if for any two points $x,y\in D$, the geodesic segment $\gamma_{x,y}(t)\in D$, $t\in [0,1]$. 
Given an arbitrary set,  $\mathcal{B} \subset M$,   the minimal convex subset that contains  $\mathcal{B}$  is called  the convex hull of $\mathcal{B}$ denoted by $ \mbox{co}(\mathcal{B})$; see \cite{jost1994} where the author assures, among other things, that \mbox{co}$(\mathcal{B})=\cup_{j=0}^{\infty}C_j$, where $C_0=\mathcal{B}$ and $C_j=\{z\in\gamma_{x,y}([0,1]): x, y \in C_{j-1}\}$.



Next result is useful and a proof can be found for example in \cite{batista2020}.
\begin{lemma}\label{eq:ContExp}
Let $\bar{p},\bar{q}\in M$ and $\{p^k\},\{q^k\}\subset M$ be such that $\lim_{k\to +\infty} p^k=\bar{p}$ and $\lim_{k\to +\infty}q^k=\bar{q}$. Then,
$\lim_{k\to +\infty}  (p^k,\exp^{-1}_{p^k},q^k)=(\bar{p},\exp^{-1}_{\bar{p}},\bar{q})$.
\end{lemma}

The next proposition contains some canonical results in the Euclidean setting that we did not identify in the literature in the Riemannian context.  As it is a basic tool which is a fundamental part in the construction of the proof of Theorem~\ref{TheoremKKM} and can be invoked in Section~\ref {4.1}.
\begin{proposition}\label{lemmaBBF2020}
Let $X\subset M$ be a convex set. Then:
\begin{itemize}
    \item [(a)] the closure of $X$ in $M$, denoted by $\overline{X}$, is a convex set;
    \item[(b)] If $x \in \emph{int}(X) $ and $y \in \overline{X}$ then $\gamma_{x,y}([0,1)) \subset \emph{int}(X)$;
    \item [(c)] the interior of $X$ in $M$, denoted by $\emph{int}(X)$ is a convex set.
\end{itemize}
\end{proposition}
\begin{proof}
Let us start by proving item a). Take $x,y\in \overline{X}$ and $z=\gamma_{x,y}(t)$, for some $t\in [0,1]$. Let $\{x_n\}, \{y_n\}\subset X$ be sequences converging to $x$ and $y$, respectively, and define the sequence $\{z_n\}\subset X$ given by $z_n=\gamma_{x_n,y_n}(t)$ for each $n\in \mathbb{N}$. Since $z_n=\mbox{exp}_{x_n}(t\mbox{exp}^{-1}_{x_n}y_n)$, for each $n\in\mathbb{N}$, the proof of item a) follows from Lemma~\ref{eq:ContExp} combined with the arbitrariness of $x, y$, $z$ and continuity of the application $\exp:TM\to M$. Now, to prove item b), suppose that there exists $0 < t <1$ such that $z = \gamma_{x,y}(t) \notin \mbox{int}(X) $. Let $u = \mbox{exp}_{z}^{-1}y$, observe that there exists $t_0 > 0$ such that $ \pi(\phi^{t_0} (z, -u)) = x$. On the other hand, $x \in \mbox{int}(X)$ so there exists an open set $U \subset X$ with $x \in U$. In particular, $\phi^{t_0} (z, -u)) \in \pi^{-1}(U)$, note that $\pi^{-1}(U)$ is an open set in $TM$. From the continuity of $\phi^{t_0}(\cdot,\cdot)$ there exists an open set $V \subset TM$ with $(z, -u) \in V$ such that $\phi^{t_0}(V) \subset \pi^{-1}(U)$. Take $\{z_n\}, \{y_n\}\subset X$ sequences converging to $z$ and $y$ respectively when $n$ goes to infinity such that $z_n \notin X$ and $y_n \in X$ for each $n$ and defines $u_n = \emph{exp}_{z_n}^{-1}y_n$ for each $n$. Note that $\|u_n\| = d(z_n, y_n) \to d(z,y)$ and, consequently, the sequence $(z_n, u_n)$ is bounded in $TM$. Passing to a subsequence, if necessary, we can suppose that $(z_n, u_n) \to (z,w)$. We claim that $w = u$. Indeed, from the continuity of the exponential application $\mbox{exp}:TM \to M$, where $\mbox{exp}(p,v) = \mbox{exp}_{p}v$, follows that $\mbox{exp}_{z_n} u_n \to \mbox{exp}_{z}w$. On the other hand, $\mbox{exp}_{z_n} u_n = y_n \to y$ and, hence, $\mbox{exp}_{z}w = y = \mbox{exp}_{z}u$. Since $\mbox{exp}_{z}(\cdot)$ is a diffeomorphism, we get $w=u$. Because $(z_n, u_n) \to (z,u)$ follows that $(z_n, -u_n) \to (z,-u)$. Thus, there is $N\in\mathbb{N}$ such that $(z_N, -u_N) \in V$ and, from the above discussion, $\phi^{t_0} (z_N, -u_N) \in \pi^{-1}(U)$. But that tells us that $p=\pi(  \phi^{t_0} (z_N, -u_N)) \in~U$ and, from the convexity of $X$, we have $\gamma_{y_N,p}([0,1])\subset X$. This is absurd considering that $z_N\in \gamma_ {y_N, p}([0.1])$ and item (b) is proved. The item (c) is a consequence of item (b). 
\end{proof}

Given $D\subset M$ convex, a function $f:D\to\mathbb{R}$ is said to be convex (resp. quasiconvex) if $f(\gamma_{x,y}(t))\leq (1-t)f(x)+tf(y)$ (resp. $f(\gamma_{x,y}(t))\leq \max\{f(x),f(y)\}$), for any $x,y\in D$ and $t\in [0,1]$.
If $D$ is also a closed set, for each $x\in M$, it is known that the projection of $x$ onto $D$ is the unique point $\bar{x}\in D$ such that $d(\bar{x},x)=d_D(y,x)$. Besides, if $\gamma_{x,y},\gamma_{x',y'}$ are two geodesics segments connecting $x,y\in M$ and $x', y'\in M$ respectively, then 
$[0,1]\ni t\mapsto d(\gamma_{x,y}(t),\gamma_{x',y'}(t))$ is a convex function. Combining these last two facts, it is not difficult to prove that $d_D(\cdot)$ is a convex function when $D$ is a convex set; see \cite[Lemma 2.5]{ledyaev2006helly}. It is also easy to prove that for each convex function $f_{\tau}:M\to \mathbb{R}$, $\tau\in \mathcal{T}:=\{1,\ldots, m\}$, the function $f(\cdot):=\max_{\tau\in\mathcal{T}}f_\tau(\cdot)$ is convex; See \cite{bento2015} for a more general case where $M$ is replaced by a convex subset of $M$ and the finite discrete set $\mathcal{T}$ is exchanged for a compact set. 
Another important class of convex functions defined over non-compact manifolds is given by Busemann functions that are defined from the distance function and, roughly, measure
the relative distance from points at infinity. Their construction goes as follows:
Let us consider a geodesic ray starting from a given point $\bar{x}$, that is, a path ${\displaystyle \gamma:[0,\infty )\to M}$ such that:
\[
d(\gamma(t),\gamma(s))=|t-s|, \qquad t,s\in [0,+\infty[.
\]
The Busemann's function ${\displaystyle b_{\gamma }:M\to \mathbb {R} }$ is defined by
${\displaystyle b_{\gamma }(y)=\lim _{t\to \infty }{\big (}d{\big (}y,\gamma (t){\big )}-t{\big )}}$. We use the notation $b_{\gamma_{x,z}}(\cdot)$ when it is intended to indicate that the ray $\gamma(\cdot)$ from $x$ passes through $z$.
Since $M\ni y\mapsto d{\big (}y,\gamma (t){\big )}-~t$ is a convex function for each $t$ fixed, it is easy to see that $b_\gamma(\cdot)$ is convex. Besides, as a limit of distance functions, $b_\gamma(\cdot)$ is Lipschitz continuous with Lipschitz constant 1. It is also not difficult to note that for each $x\in M$ fixed, $[0,+\infty[\ni t\mapsto \psi_y(t)= d(y,\gamma (t))-t$ is non-increasing, $\psi_y(\cdot)$ is bounded and, in particular, for $t=d(z,x)$, we have: 
\[
b_{\gamma_{z,x}}(y)\leq d(y,x)-d(z,x). 
\]
For a good discussion and examples of Busemann's functions in some specific Hadamard manifolds see, for example, \cite{bridson2013metric}.  These functions, which were initially introduced by Herbert Busemann to define the parallel axiom on a certain class of metric spaces, see \cite {busemann1955}, have been explored as a tool in important literature for various other purposes; see for example \cite{busemann1993,li1987positive,10.4310/jdg/1214460863,shiohama1979busemann} and references therein. In this present paper the Busemann's function is used to introduce an alternative resolvent associated with equilibrium problems which can be seen as a first step in responding to a problem posed in \cite[Section 5]{cruz2017note}. 

The next lemma is a result not found in the literature that is used for the purposes of the paper related to the well-definedness of the aforementioned resolvent.
\begin{lemma}\label{lemmaUpperSem}
Let $\triangle := \{(x,x): x \in M\}$. Then, $b:M \times (M \backslash \triangle) \times M \to \mathbb{R}$, given by $b(x, z, y)= b_{\gamma_{x,z}}(y)$, is an upper semi-continuous function.
\end{lemma}
\begin{proof}
Take $x,z,y \in M$, $x \neq z$. Given $\epsilon > 0$ there is $t_0 > 0$ such that $b_{\gamma_{x,z}}(y) \leq d(\gamma_{x,z}(t_0), y) - t_0 < b_{\gamma_{x,z}}(y) + ~\frac{\epsilon}{2}$. Note that $$\pi (\phi^{t_0}(x, d(x,z)^{-1}\mbox{exp}_{x}^{-1}z)) = \gamma_{x,z}(t_0).$$
From the continuity of $\phi^{t_0}(\cdot)$ there exists an open set $V_1 \subset TM$ with 
$$ (x, d(x,z)^{-1}\mbox{exp}_{x}^{-1}z) \in V_1$$ such that $\phi^{t_0}(V_1) \subset \pi^{-1}(B(\gamma_{x,z}(t_0), \frac{\epsilon}{4})$. We claim that there exist two open disjoint sets $U_1$ and $U_2$ such that $x \in U_1$, $z \in U_2$ and $(x', d(x',z')^{-1}\mbox{exp}_{x'}^{-1}z') \in V_1$ for any $x' \in U_1, z'\in U_2$. In fact, if this case does not hold, then there exist sequences $\{z_n\}, \{x_n\}\subset M$ converging to $z$ and $x$, respectively, when $n$ goes to infinity such that $(x_n, d(x_n,z_n)^{-1} \mbox{exp}_{x_n}^{-1}z_n) \notin V_1$ for each $n$. By using an argument analogous to that considered in Proposition~\ref{lemmaBBF2020}, we can conclude that there is $(x_{n_k}, d(x_{n_k},z_{n_k})^{-1} \mbox{exp}_{x_{n_k}}^{-1}z_{n_k}) \to (x,d(x,z)^{-1} \mbox{exp}_{x}^{-1}z)$. But this is an absurd because $V_1$ is open. Now, let us consider $x' \in U_1, z' \in U_2 $, $y' \in B(y, \frac{\epsilon}{4})$ and note that
\begin{equation}\label{ineq.UpperS}
d(\gamma_{x',z'}(t_0), y') - t_0 \leq  d(\gamma_{x',z'}(t_0), \gamma_{x,z}(t_0)) + d(y, y') + d( \gamma_{x,z}(t_0), y)- t_0.
\end{equation}
On the other hand, $(x', d(x',z')^{-1}\mbox{exp}_{x'}^{-1}z') \in V_1$ and, hence, 
$$\gamma_{x',z'}(t_0) = \pi (\phi^{t_0}(x, d(x',z')^{-1}\mbox{exp}_{x'}^{-1}z') \in        B\left(\gamma_{x,z}(t_0), \frac{\epsilon}{4}\right). $$
Therefore, from the inequality in (\ref{ineq.UpperS}) follows that
$$b_{\gamma_{x',z'}}(y)' \leq  d(\gamma_{x',z'}(t_0), y') - t_0 < b_{\gamma_{x,z}}(y) + \epsilon,$$
and the proof is concluded.  
\end{proof}

Given a nonempty and convex set $D\subset M$ and a real valued function $f:~D\to\mathbb{R}$ we denote its epigraph by 
\begin{equation}\label{epi1}
\mbox{epi}(f):=\{(x,\beta)\in D\times\mathbb{R} : f(x) \leq \beta\}.
\end{equation}

Next proposition can be found for example in \cite{ferreira2002,sturm2003probability} in the particular case where $D=M$. However, taking into account that $D$ is a convex set, its proof in which case $D\neq M$ remains a consequence of the fact that $\tilde{\gamma}=(\gamma^1,\gamma^2)$ is a geodesic of $M\times\mathbb{R}$ if and only if $\gamma^1$ is a geodesic of $M$ and $\gamma^2$ is a geodesic of $\mathbb{R}$; see, for example, \cite{udriste1994convex}.
\begin{proposition}\label{epi-convex}
Let $D\subseteq M$ be a nonempty and convex set. Then, a function $f:D\to\mathbb{R}$ is convex if only if $\emph{epi}(f)$ is a convex set. 
\end{proposition}

\begin{definition}\label{Def ConvexCombination}
Let $D\subset M$ be a nonempty set, $x_i\in D$, $\alpha_i\geq 0$, $i=1,\ldots,N$ such that $\sum_{i=1}^N \alpha_i=1$ and consider a finite sequence $t_1,t_2,\ldots, t_{2N-2}$ given as follows:
\[
t_{2k}=\frac{\alpha_{k+1}}{\sum_{i=1}^{k+1} \alpha_i}\quad \emph{and}\quad t_{2k-1}=\frac{\sum_{i=1}^{k} \alpha_i}{\sum_{i=1}^{k+1} \alpha_i},\quad k\in\mathbb{N},\quad 1\leq k\leq N-1.
\]
Given the sequence $\{y_1,\ldots,y_N\}$ where $y_1=x_1$ and 
$y_k=\gamma_{y_{k-1},x_k}(t_{2(k-1)})$, $k\geq 2$, $y_N$ is the convex combination of elements $x_1,\ldots, x_N$ belonging to $D$ denoted by 
\begin{equation}\label{ConvexCombination1}
\emph{comb}[x_1(\alpha_1),x_2(\alpha_2),\ldots, x_N(\alpha_N)].
\end{equation}
\end{definition}
\begin{remark}\label{remark1}
Note that the last definition can be extended to any ordering of the points $x_1,\ldots, x_N$, which is determined by choosing one of the $N!$ possibilities of permutations of $\{1,\ldots, N\}$. It is easy to see that the definition of $ y_N$ above is non-commutative in the sense that it depends on the chosen permutation. This notion of convex combination, known as ``pseudo-convex combination" was introduced in \cite{zhou2019}. Hence, from \cite[Theorem 3.1]{zhou2019} we can conclude that for $x_1,\ldots, x_N\in D$, then $\mbox{comb}[x_1(\alpha_1),x_2(\alpha_2),\ldots, x_N(\alpha_N)] \in ~D$. 
\end{remark}

The following is a definition of commutative convex combination.
\begin{definition}\label{defConvexCombination}
Let $D\subset M$ be a nonempty set, $x_i\in D$, $\alpha_i\geq 0$, $i=~1,\ldots,N$ such that $\sum_{i=1}^N \alpha_i=1$ and $\mathcal{P}_N$ the set of all permutations of $1,\ldots,N$. For a permutation $a=(j_1,\ldots,j_N)\in \mathcal{P}_N$, let us consider the convex combination of $x_{j_1},\ldots, x_{j_N}$ given by $\mbox{comb}[x_{j_1}(\alpha_{j_1}),x_{j_2}(\alpha_{j_2}),\ldots, x_{j_N}(\alpha_{j_N})]=y(a)$ and the following probability measure 
\begin{equation}\label{medidaprob}
\mu=\frac{1}{N!}\sum_{a\in \mathcal{P}_N}\delta_{y(a)},
\end{equation}
where $\delta_{y(a)}$ denotes for the Dirac measure at the point $y(a)$. The commutative convex combination is given as follows:
\begin{equation}\label{barycenter1}
   b(\mu):=\emph{argmin}_{z\in M}\frac{1}{N!}\sum_{a\in \mathcal{P}_N}d^2(z,y(a)).
\end{equation}
\end{definition}
The problem in (\ref{barycenter1}) is  well defined because $d^2(\cdot,y(a))$ is a $1$-coercive and strictly convex function for each $a\in \mathcal{P}_N$; see, for example, \cite{ferreira2002}. 
It is worth noting that $b(\cdot)$ in (\ref{barycenter1}), known as the Riemannian center of mass or Karcher average due to \cite{grove1973}, has been extensively studied in pure mathematics as well as applied fields, see \cite{grove1974,grove1974jacobi, grove1976,karcher1977,kristaly2008,bini2013,moakher2005,aastrom2017image} and their references therein.
For algorithms used in its computation see, for example, \cite{afsari2013,bacak2014computing, Bento2019}. 

\subsubsection{Jensen's Inequality}
In this section, Jensen's inequality is introduced in the Riemannian context taking into account the convex combinations introduced in Section~\ref{section2}. This important inequality, attributed to the Danish mathematician Johan Jensen due to the paper \cite{jensen1906}, has appeared in the nonlinear setting in the case in which the ``convex combination" of the points involved is the ``center of mass" (barycenter) of the points both in the discrete case and in the continuous case associated with a measure of probability; see for example \cite{jost1994,sturm2003probability,Bacak2014:1} and their references therein.

In the next two results we assume that $D\subset M$ is a non-empty and convex set, and $f:D\to\mathbb{R}$ is a convex function.
\begin{theorem}\label{Jensen1}
For any $N\in\mathbb{N}$, $x_i\in D$ and $\alpha_i\geq 0$, $i=1,\ldots,N$ such that $\sum_{i=1}^N \alpha_i=1$. If $y_N$ is given as in \eqref{ConvexCombination1}, then:
\begin{equation}\label{Jensen1.1}
f(y_N)\leq \sum_{i=1}^N \alpha_if(x_i).
\end{equation}
\end{theorem}
\begin{proof}
Take $x_1,\ldots, x_N\in D$. From Remark~\ref{remark1}, we have $y_N\in D$. In particular, using definition of the epigraph of $f$ introduced in (\ref{epi1}), it follows that $(y_N,f(y_N))\in \mbox{epi} f$. On the other hand, using again definition of \mbox{epi}$f$, we have $(x_i,f(x_i))\in \mbox{epi} f$ for $i=1,\ldots, N$. Using convexity of $f$ and applying Proposition~\ref{epi-convex}, we conclude that $\tilde{y}_N=\mbox{comb}[(x_1,f(x_1))(\alpha_1),\ldots, (x_N,f(x_N))(\alpha_N)]$. Now, taking into account that  
$\tilde{\gamma}:[0,1]\to M\times\mathbb{R}$ is a geodesic if, only if, $\tilde\gamma(t)=(\gamma^1(t),\gamma^2(t))$ where $\gamma^1:[0,1]\to M$ and $\gamma^2:[0,1]\to\mathbb{R}$ are geodesics, we have $\tilde{y}_N=\left(y_N,\sum_{i=1}^N\alpha_if(x_i)\right)$ and the desired result it follows from the definition of \mbox{epi}$f$.  
\end{proof}

\begin{theorem}
For any $N\in\mathbb{N}$, $x_i\in D$ and $\alpha_i\geq 0$, $i=1,\ldots,N$ such that $\sum_{i=1}^N \alpha_i=1$. If $b(\cdot)$ is given as in (\ref{barycenter1}), then:
\begin{equation}\label{JensenIneq}
f(b(\mu))\leq \sum_{i=1}^N \alpha_if(x_i).
\end{equation}
\end{theorem}
\begin{proof}
From \cite[Proposition 2.3.8]{Bacak2014:1} combined with the probability measure in (\ref{medidaprob}), we have 
\begin{equation}\label{JensenBacak}
f(b(\mu))\leq \frac{1}{N!}\sum_{a\in \mathcal{P}_N}f(y(a)),
\end{equation}
where $``a"$ and $``y(a)"$ are given in Definition~\ref{defConvexCombination}. On the other hand, taking into account Remark~\ref{remark1}, we can apply Theorem~\ref{Jensen1} to the convex combination obtained through permutation $a=(j_1,\ldots,j_N)\in \mathcal{P}_N$ obtaining the following variant of (\ref{Jensen1.1}):
\begin{equation}\label{Jensen2.2}
f\left(y(a)\right)\leq \sum_{i=1}^N \alpha_{j_i}f(x_{j_i}).
\end{equation}
Therefore, the inequality in (\ref{JensenIneq}) follows immediately by combining (\ref{JensenBacak}) and (\ref{Jensen2.2}), which concludes the proof of the theorem.  
\end{proof}

\subsubsection{Carathédory's theorem}
One of the pillars of combinatorial convexity is the so-called Carathéodory's theorem introduced by Constantin Carathéodory in the linear setting in \cite{caratheodory1907variabilitatsbereich} but, as noted in \cite[Remark~3.2]{zhou2019}, its validity is not known even in Hadamard manifolds. The absence of a Carathéodory theorem makes us face some  significant obstacle as, for example, ``what is the convex hull of three points in a 3 or higher dimensional Riemannian manifold?" This issue was highlighted by Berger in \cite[page 253]{berger2012panoramic}, who also conjectured that said convex hull is not closed, except in very special cases. As the conjecture is placed in the case where $M$ is 3 or higher dimensional Riemannian manifold, we see  the case two-dimensional is a folklore result. We present below a brief discussion involving the necessary elements for an induction proof of the Carathéodory's theorem where we indicate some specific steps by way of illustration only. Let $M$ be a two-dimensional Hadamard manifold, i.e., n = 2. Given $y\in M$ and $v\in T_yM$ we introduce the following notion of Riemanniann semi-space:
\[
S_{v}:=\left\{x\in M:\langle v,\exp^{-1}_yx\rangle\leq0\right\}.
\]
It is clear that $S_v$ is closed and, from \cite{ferreira2005} it is also convex in the particular case where $M$ has  constant curvature (as noted in \cite{batista2020}, so far it is not known if $S_v$ is or not convex in general Hadamard manifolds). However, it is not difficult to prove that in the case $n=2$ the convexity follows even for non-constant curvature and also that the geodesic $\gamma_{v^\perp}(\cdot,y)$  divides $M$ into two convex parts represented by $S_v$ and $S_{-v}$. Given $\mathcal{B}:=\{x_1,\ldots,x_m\}\subset M$, $m\in\mathbb{N}$, then any $x\in \mbox{co}(\mathcal{B})$ can be written in terms of no more than $n+1$ points. Note that in the case where $m = 1$ or $2$ there is nothing to do. Just as an illustration to clarify the procedure, let us build the cases:
\begin{itemize}
    \item [a)] $m=3$;
    \item [b)] $m=4$.
\end{itemize}
Assume that happens $a)$, $ x_1, x_2, x_3 $ are non-collinear points (otherwise we would be in the case $m = 2$ already considered) and let us consider the triangle $\Delta_{x_1, x_2, x_3}$ given by the intersection of the semi-spaces $S_{v_1},S_{v_2},S_{v_3}$ where $v_1=(\mbox{exp}^{-1}_{x_1}x_2)^\perp$, $v_2=(\mbox{exp}^{-1}_{x_2}x_3)^\perp$ and $v_3=(\mbox{exp}^{-1}_{x_3}x_1)^\perp$. From the definition of $\mbox{co}(\mathcal{B})$ and taking into account that $\Delta_{x_1, x_2, x_3}$ is a convex set, we have $\mbox{co}(\mathcal{B})\subset \Delta_{x_1, x_2, x_3}$. Thus, for a given $x\in \mbox{co}(\mathcal{B})$, either $x\in \partial (\Delta_{x_1, x_2, x_3})$ (border of $\Delta_{x_1, x_2, x_3}$) or $x\in \mbox{int}(\Delta_{x_1, x_2, x_3})$ (interior of $\Delta_{x_1, x_2, x_3}$). On the one hand, if $x\in \partial (\Delta_{x_1, x_2, x_3})$, $x=\mbox{exp}_{x_i}tx_j$ for some $i,j\in\{1,2,3\}$ and $t\in[0,1]$. On the other hand, if $x\in \mbox{int}(\Delta_{x_1, x_2, x_3})$ then $x=\mbox{exp}_{x_1}t_1(\mbox{exp}^{-1}_{x_1}(\mbox{exp}_{x_2}t_2(\mbox{exp}^{-1}_{x_2}x_3)$ for $t_1,t_2\in [0,1]$ ensuring that the result is true for $m = 3$; see characterization of the convex hull presented in Section~\ref{Sec2.2}. Let us suppose now that happens $b)$ and assume that any three points are non-collinear. Without loss of generality let us consider the cases where either $x_4\in 
\Delta_{x_1, x_2, x_3}$ or $x_4\notin \Delta_{x_1, x_2, x_3}$. If the first case happens, then using arguments similar to the previous one, it is possible to conclude that $\mbox{co}(\mathcal{B})=\mbox{co}(\{x_1,x_2,x_3\})$ and the result follows. Otherwise, if the second case happens, we can again without loss of generality assume that $x_4\in S_{-v_1}$. That being the case, $\mathcal{B}$ determines the following convex and closed set $\cap_{i=1}^4S_{v_i}$, where  $v_1=(\mbox{exp}^{-1}_{x_1}x_4)^\perp$, $v_2=(\mbox{exp}^{-1}_{x_4}x_2)^\perp$, $v_3=(\mbox{exp}^{-1}_{x_2}x_2)^\perp$, 
$v_4=(\mbox{exp}^{-1}_{x_3}x_1)^\perp$ and, consequently, $\mbox{co}(\mathcal{B})\subset \cap_{i=1}^4S_{v_i}$. The geodesic $\gamma_{v}(\cdot,x_1)$, where $v=\mbox{exp}^{-1}_{x_1}x_2$, divides the quadrilateral $\cap_{i=1}^4S_{v_i}$ into two closed and convex triangles, namely, $\Delta_{x_1, x_2, x_4}\subset S^{x_1}_{v^\perp}$ and $\Delta_{x_1, x_2, x_3}\subset S^{x_1}_{-v^\perp}$. But that tells us that any
$x\in \mbox{co}(\mathcal{B})$ can be written in one of the following ways: either $x=\mbox{exp}_{x_i}tx_j$ for some $i,j\in\{1,2,3,4\}$ and $t\in[0,1]$ (this is the case when $x$ belongs to one of the sets $\partial (\Delta_{x_1, x_2, x_4})$, $\partial (\Delta_{x_1, x_2, x_3})$  or $\gamma_{x_1,x_2}([0,1])$) or $x=\mbox{exp}_{x_1}t_1(\mbox{exp}^{-1}_{x_1}(\mbox{exp}_{x_2}t_2(\mbox{exp}^{-1}_{x_2}x_3)$ (resp. $x\in \mbox{exp}_{x_1}t_1(\mbox{exp}^{-1}_{x_1}(\mbox{exp}_{x_2}t_2(\mbox{exp}^{-1}_{x_2}x_4)$) for $t_1,t_2\in [0,1]$ if $x\in \mbox{int}(\Delta_{x_1, x_2, x_3})$ (resp.  $x\in \mbox{int}(\Delta_{x_1, x_2, x_4}))$ ensuring that the result is true for $m=4$. 

\section{Helly's theorem and KKM lemma}\label{section3}
In this section, our main focus is explore Helly’s theorem from theoretical viewpoint and obtain an alternative version for Knaster-Kuratowski-Mazurkiewicz theorem, also known as the KKM lemma, suitable for our purposes in the next section dealing with existence result of solutions for equilibrium problems.

\subsection{Helly's theorem}
 The Helly's theorem, introduced by Edward Helly in the linear setting (see \cite{helly1923}), is an important result from convex geometry which gives sufficient conditions for a family of convex sets to have a nonempty intersection. Over the years a large variety of proofs as well as applications have been presented; see, for example, \cite{eggleston1958,danzer1963helly,lay2007} where its relation with another important classic results of convex geometry can also be found.

As far as we know, the first approach to Helly's theorem in the Riemannian context was presented in \cite{ledyaev2006helly} where, specifically, the authors generalize the classical Helly’s theorem about the intersection of convex sets in $\mathbb{R}^n$ for the case of manifolds of nonpositive curvature (for example, Hadamard manifolds). The main result in \cite{ledyaev2006helly}, identified in the referred paper as Theorem~4.2, is described below:

\begin{theorem}\label{mainLedyaev}
Let $M$ be an $n-$dimensional $C^{\infty}$ Riemannian manifold with nonpositive curvature and let $K$ be an open convex subset of $M$ satisfying  the $(CC)-$condition. Let $\{C_a\}_{a\in A}$ be a family of closed convex subsets of $K$ and let at least one of them be compact. Suppose that, for any $n+1$ elements $a_1,\ldots, a_{n+1}\in A$,
\[
\cap_{j=1}^{n+1}C_{a_j}\neq\emptyset.
\]
Then, $\cap_{a\in A}C_{a}\neq\emptyset.$
\end{theorem}
Note that the above result is obtained in the case where $K$ satisfies a certain $(CC)-$condition which ensures that the convex hull of any finite set of points in $K$ is a compact set. 
\begin{remark}\label{remarkconvexball}
If $M$ is a Hadamard manifold, for a given point $\bar{p}\in M$ and $0<r<+\infty$, a closed ball $\overline{B(\bar{p},r)}:=\{p: d(p,\bar{p})\leq r\}$ is a convex and compact set. In this context, it follows that the convex hull of a given compact set $D$ is necessarily non-empty and bounded.
However, taking into account the approach in the linear setting and considering the discussion with an emphasis on Berger's conjecture addressed in Section 2.3, we highlight the following difficulties to guarantee the $ CC-$condition as defined in \cite{ledyaev2006helly}:
\begin{itemize}
\item [a)] how to ensure that each element of \mbox{co}$(D)$ is in fact expressed as a ``convex combination" of points belonging to D? As mentioned in \cite[Remark~3.3]{zhou2019}, this does not necessarily happen in general;
\item [b)] as far as we know, in the particular case where $M=\mathbb{R}^n$ to show that \mbox{co}$(D)$ is closed it was necessary to use that each element of \mbox{co}$(D)$ can be written as a convex combination of no more than $\mbox{dim}(M)+1$ points from set $D$. This is the content of the so-called Carathéodory's theorem that, as noted in \cite[Remark~3.2]{zhou2019}, its validity is not known even in Hadamard manifolds.
\end{itemize}
\end{remark}
Next, we present an alternative proof for Theorem~\ref{mainLedyaev} in the case where $M$ is an $n$-dimensional Hadamard manifold without admitting the $CC-$condition as a assumption. 
\begin{theorem}\label{Hellymain1}
Let $\mathcal{A}$ be a set of indices, $\{B_a\}_{a\in \mathcal{A}}$ a family of closed and convex sets and assume that there exists $a'\in \mathcal{A}$ such that $B_{a'}$ is compact. If intersection of any $n+1$ sets of the family $\{B_a\}_{a\in\mathcal{A}}$,
\[
\cap_{j=1}^{n+1}B_{a_j}\neq\emptyset,
\]
then $\cap_{a\in \mathcal{A}}B_{a}\neq\emptyset.$
\end{theorem}
\begin{proof}
The proof is divided into three cases, namely, when
\begin{enumerate}
\item [a)] $\mathcal{A}=\{1,\ldots,m\},\quad m>n+1$;
\item [b)] $\mathcal{A}$ is a set of infinite indices.
\end{enumerate}
Let us start with item a). Given $a_1, \ldots, a_{n+1}\in\mathcal{A}$, by hypothesis we can take $p_{a_1\ldots a_{n+1}}\in \cap_{j=1}^{n+1}B_{a_j}\neq\emptyset$. Because $\mathcal{A}$ has a finite number of elements, it is clear that we can choose $\bar{p}\in M$ and $0<r<+\infty$ such that  
$p_{a_1\ldots a_{n+1}}\in \overline{B(\bar{p},r)}$, for any $a_1,\ldots a_{n+1}\in \mathcal{A}$.
From what has already been noted in  Remark~\ref{remarkconvexball}, we can conclude that $\overline{B(\bar{p},r)}$ is a convex and compact set. Let us define $\tilde{C}_a:=B_a\cap \overline{B(\bar{p},r)}$, $a\in \mathcal{A}$. Therefore, taking into account that for any $n+1$ sets $B_a$ one have  $\cap_{j=1}^{n+1}B_{a_j}\neq\emptyset$, the desired resulted follows directly by applying \cite[Theorem 4.1]{ledyaev2006helly} with $K=M$ (remember that Hadamard manifolds are convex), $C=\overline{B(\bar{p},r)}$, $C_a=\tilde{C}_a$, $a\in \mathcal{A}=\{1,\ldots, m\}$.
For proving the item b), let us define:
\begin{eqnarray*}
J&:=&J_1\times J_2\times\ldots \times J_n,\qquad J_1=J_2=\ldots=J_n:=\mathcal{A}\setminus\{a'\},\\
J^*&:=&\{\beta=(a_1,\ldots, a_n)\in J~:~ a_i\neq a_j, i\neq j\},\\
A_\beta &:=& B_{a'}\cap_{j=1}^nB_{a_j},\quad \beta\in J^*.
\end{eqnarray*}
Taking into account that $A_\beta$ is the intersection of $n+1$ sets of the family $\{B_a\}_{a\in\mathcal{A}}$, it follows immediately from the hypothesis that $A_\beta\neq\emptyset$ for each $\beta\in J^*$. Note that, for any $A_{\beta_1},\ldots, A_{\beta_{n+1}}$, from item a) it follows that  $\cap_{j=1}^{n+1}A_{\beta_j}\neq\emptyset$.
Therefore, the conclusion of the proof goes on by applying \cite[Theorem 4.1]{ledyaev2006helly} with $K=M$, $C=B_{a'}$, $C_a=A_\beta$ and by considering that  $\cap_{a\in \mathcal{A}}B_{a}=\cap_{\beta\in J^*}A_{\beta}$.    
\end{proof}

\subsection{KKM lemma}
The KKM lemma is associated with fixed point theory and was published in 1929 by the three Polish mathematicians Knaste, Kuratowski, Mazurkiewicz; see \cite{Knaster1929}. A brief discussion of this important result of convex geometry was presented in the introduction to the paper. In this section, we use Helly's theorem to obtain the following version of the KKM lemma:
\begin{theorem}\label{TheoremKKM}
Let $K\subset M$ and $G:K\to 2^K$ a mapping such that, for each $x\in K$, $G(x)$ is closed and convex set. Suppose that 
\begin{enumerate}
    \item [(a)] there exists $x_0\in K$ such that $G(x_0)$ is compact;
    \item [(b)] for any $x_1,\ldots,x_{n+1}\in K$, \emph{co}$(\{x_1,\ldots,x_{n+1}\})\subset \cup_{i=1}^{n+1}G(x_i)$. 
\end{enumerate}
Then, $\cap_{x\in K}G(x)\neq\emptyset$.
\end{theorem}
\begin{proof}
Given $m\in\mathbb {N}$, $m\leq n+1$, define $I_m=\{1,\ldots,m\}$ and $B_m=\{x_i :i\in I_m\}$. We claim that 
\begin{equation}\label{KKMtheorem}
\left(\cap_{i\in I_m}G(x_i)\right)\cap\overline{\mbox{co}({B_m})}\neq\emptyset.
\end{equation}
Indeed, first of all note that from assumption (b) this fact is true for $m = 1$. Now, following the inductive process, let us assume the statement (\ref{KKMtheorem}) is true for any set containing $m-1$ elements and take 
\begin{equation}\label{inclusionTec}
\tilde{x}_{j}\in\left(\cap_{i\in I_m\setminus\{j\}}G(x_i)\right)\cap\overline{\mbox{co}({B_m}\setminus\{x_{j}\})}, \qquad j\in I_m.
\end{equation}
Taking $\tilde{B}_m:=\{\tilde{x}_j:j\in I_m\}$, it follows that \mbox{co}$(\tilde{B}_m)\subset \overline{\mbox{co}(B_m)}$. Now, given $r=\max_{s\in I_m\setminus\{1\}}\{d(\tilde{x}_1,\tilde{x}_s)\}$ and taking into account that $B(\tilde{x}_1,r)$ is a convex set in Hadamard manifolds, from the definition of convex hull it is easy to see that $\mbox{co}(\tilde{B}_m)\subset B(\tilde{x}_1,r)$ and, consequently, $\overline{\mbox{co}(\tilde{B}_m)}$ is a convex and compact set.  Since $\overline{\mbox{co}(\tilde{B}_m)}\subset\overline{\mbox{co}(B_m)}$, to conclude the statement it is sufficient to prove that $\left(\cap_{i\in I_m}G(x_i)\right)\cap\overline{\mbox{co}({\tilde{B}_m})}\neq\emptyset.$ Let us assume, by contradiction, that 
\begin{equation}\label{equalitytec}
\left(\cap_{i\in I_m}G(x_i)\right)\cap\overline{\mbox{co}({\tilde{B}_m})}=\emptyset.
\end{equation}
Using Proposition~\ref{lemmaBBF2020} with $X=\mbox{co}({\tilde{B}_m})$ it follows that $\overline{\mbox{co}({\tilde{B}_m})}$ is convex. On the other hand, because $G(x)$ is a closed and convex set for each $x\in K$, we obtain that:
\begin{enumerate}
    \item [(i)] $\overline{\mbox{co}({\tilde{B}_m})}\cap G(x_i)$ is a closed and convex set for each $i\in I_m$;
    \item [(ii)] $M\ni x\mapsto\psi_i(x):=d(x,{G(x_i)\cap \overline{\mbox{co}({\tilde{B}_m})}})$ is convex for each $i\in I_m$;
    \item [(iii)] $M\ni x\mapsto\psi(x):=\max_{i\in I_m}\psi_i(x)$ is convex.
\end{enumerate} 
Since $\overline{\mbox{co}({\tilde{B}_m}})$ is a compact set, there exists $\hat{x}\in\mbox{argmin}\{\phi(x) : {x\in \overline{\mbox{co}({\tilde{B}_m}})}\}$ and, by combining (\ref{equalitytec}) with $(i)$ and definition of $\psi_i(\cdot)$ and $\psi(\cdot)$ in $(ii)$ and $(iii)$ respectively, it follows  that $\psi(\hat{x})>0$. Note that $\overline{\mbox{co}({\tilde{B}_m)}}\subset\overline{\mbox{co}(B_m)}$, by construction, and, by using assumption (b), $\mbox{co}(B_m)\subset \cup_{i=1}^{m}G(x_i)$. In particular,  $\overline{\mbox{co}({\tilde{B}_m)}}\subset \cup_{i=1}^{m}G(x_i)$ and, consequently, from the definition of $\hat{x}$ it follows that there exists $i_0\in I_m$ such that $\hat{x}\in G(x_{i_0})$. Now, take $\tilde{x}_{i_0}\in \tilde{B}_m\subset \overline{\mbox{co}({\tilde{B}_m)}}$ and consider the geodesic segment $\gamma_{\hat{x},\tilde{x}_{i_0}}([0,1])\subset \overline{\mbox{co}({\tilde{B}_m)}}$. Using convexity of $\psi_{i_0}$ and taking into account that $\psi_{i_o}(\hat{x})=0$ (this is because $\hat{x}\in \overline{\mbox{co}({\tilde{B}_m)}}\cap G(x_{i_0})$), for each $t\in [0,1]$, we have $\psi_{i_0}(\gamma_{\hat{x},\tilde{x}_{i_0}}(t))\leq t\psi_{i_0}(\tilde{x}_{i_0})$. Hence, $\psi_{i_0}(\gamma_{\hat{x},\tilde{x}_{i_0}}(t))$ tends to zero as $t$ goes to zero and, using again that $\psi(\hat{x})>0$, in particular, there exists $\tilde{t}$ sufficiently close to zero, we get \begin{equation}\label{inequaltyTec2}
\psi_{i_0}(\gamma_{\hat{x},\tilde{x}_{i_0}}(\tilde{t}))< \psi(\hat{x}).
\end{equation}
Now, take $i\in I_m\setminus\{i_0\}$ and note that, by using (\ref{inclusionTec}) with $j=i_0$, from the definition of $\psi_i(\cdot)$ in (i) it follows that $\psi_i(\tilde{x}_{i_0})=0$. Thus, convexity of $\psi_i(\cdot)$ implies that 
$\psi_{i}(\gamma_{\hat{x},\tilde{x}_{i_0}}(\tilde{t}))< (1-\tilde{t})\psi_{i}(\hat{x})<\psi(\hat{x})$. By combining the latter inequality with (\ref{inequaltyTec2}), and taking into account the definition of $\psi(\cdot)$ in (ii), we conclude that $\psi(\gamma_{\hat{x},\tilde{x}_{i_0}}(\tilde{t}))<\psi(\hat{x})$, which contradicts the fact that $\hat{x}\in\mbox{argmin}\{\phi(x) : {x\in \overline{\mbox{co}({\tilde{B}_m}})}\}$. Therefore, the desired result follows by using Theorem~\ref{Hellymain1}. 
\end{proof}
\begin{remark}\,
\begin{enumerate}
\item [(i)]
The construction of the proof of the previous theorem followed the same idea explored  in the linear context. In any case, we chose to present it in detail in order to make clear to the reader where some adjustments were necessary;
\item [(ii)] Others versions of the KKM lemma can be found, for example, in \cite{colao2012equilibrium,zhou2019,park2019riemannian} where, in item (b), $n + 1$ is replaced by a certain variable natural $m$. However, since this topic dealing with combinatorial convexity shows to have an interdisciplinary character, the version of the KKM lemma presented in the last result seems more appropriate, inclusively, for our purposes in the next section dealing with existence result of solutions for equilibrium problems.  \end{enumerate}
\end{remark}

\section{Equilibrium problem}\label{section4}
In this section, we explore the KKM's lemma to establish an existence result of solutions for equilibrium problems and to ensure the well-definedness of the resolvent and, in particular, of the proximal point algorithm for solving equilibrium problems.
\subsection{Existence for equilibrium problem}\label{4.1}
As already highlighted in the introduction of the paper,  the KKM lemma was used as a tool to establish result of existence for equilibrium problems; see, for example, \cite{colao2012equilibrium,park2019riemannian,bento2021,batista2015existence} for references dealing with this topic in the Riemannian setting, whose approaches extend the results of existence established directly to some important particular instances such as variational inequality \cite{nemeth2003variational,li2009existence,li2012variational} and Nash equilibrium points \cite{kristaly2010location,kristaly2014nash}.
Limiting the reference \cite{bento2021}, which we believe to be the most recent on the topic, it is possible to notice an important connection between combinatorial convexity, established by the KKM lemma, and “variational rationality” approach of human behavior,  characterizing the relevance of the theme to the interdisciplinary research. 

Next, we recall the general equilibrium problem. Given $\Omega \subset M$ a nonempty closed convex set and a bifunction $F: \Omega\times \Omega \rightarrow \mathbb{R}$ satisfying the property $F(x,x) =0$, for all $x \in \Omega$, the  equilibrium problem in the Riemannian context (denoted by EP) consists in:
\begin{equation}\label{ep}
\mbox{Find $x^{*}\in \Omega$}:\quad F(x^\ast,y) \geq 0,\qquad y \in \Omega.
\end{equation}
As far as we know, (\ref{ep}) was first introduced in the Riemannian setting in \cite{nemeth2003variational} in the particular case where $F(x,y):= \langle V(x), \exp^{-1}_{x}y \rangle$,  $x, y\in \Omega$, for $V(\cdot)$ being a single-valued vector field on Hadamard manifolds.
The existence result for (\ref{ep}) established in \cite{bento2021} took into account the following assumptions:
\begin{enumerate}
    \item [(i)] $F(\cdot,\cdot)$ is pseudomonotone, i.e.,  for each $(x,y)\in \Omega\times \Omega$, $F(x,y)\geq 0$ implies $F(y,x)\leq 0$;
    \item[(ii)] For every $x\in \Omega$,  $y\mapsto F(x,y)$ is convex and lower semicontinuous;
	\item[(iii)] For every $y\in \Omega$,  $x\mapsto F(x,y)$ is upper semicontinuous;
	\item[(iv)] Given $z_{0}\in M$ fixed, consider a sequence  $\{z^k\} \subset \Omega$ such that $\{d(z^k,z_0)\}$ converges to infinity as $k$ goes to infinity.
	Then, there exists $x^\ast \in \Omega$ and $k_0 \in \mathbb{N}$ such that 
	\[
	F(z^k, x^\ast)\leq 0, \qquad k\geq k_0;
	\]
	\item[(v)]
	Given $k\in \mathbb{N}$, for all finite set $\{y_1,\ldots, y_m \}\subset \Omega_{k}$, one has 
	\[
	\mbox{co} ( \{ y_1,\ldots, y_m\})\subset\bigcup_{i=1}^mL_{F}(k,y_i),
	\]
	$L_F (k, y) := \{x\in \Omega _k : F(y, x)\leq 0\}$ and $\Omega_k:=\{x\in \Omega:\; d(x,z_0)\leq k\}$.
\end{enumerate}
Since in the Existence Result what it is really need is the convexity of the set $L_F (k, y)$ for each $y\in\Omega$, then assumption (ii) can be exchanged for:
\begin{enumerate}
 \item[($ii^*$)] For every $y\in \Omega$,  $\{x\in \Omega: F(y,x)<0\}$ is convex and $y\mapsto F(x,y)$ is lower semicontinuous.
\end{enumerate}
Note that $(ii^*)$ naturally holds when $F(y,\cdot)$ is convex or quasiconvex for $y\in \Omega$.
In addition, as noted in \cite[Remark~5]{bento2021}, assumption (i) is a sufficient condition for happening assumption (v), which is true even in the case where $m=n+1$. Indeed, let us consider $y_1,\ldots, y_{n+1}\in \Omega_k$,  take $\bar{y}\in \mbox{conv}(\{y_{1},\ldots, y_{n+1}\})$ and let us suppose, for contradiction, that $\bar{y}\notin \bigcup_{i=1}^{n+1}L_{F}(k,y_i)$. But that tells us that,
	\begin{equation}\label{eq:equil1}
	F(y_{i},\bar{y})>0, \qquad i\in \{1,\ldots, n+1\}.
	\end{equation}
	Now, define the following set $B:=\{x\in \Omega_{k}:\; F(\bar{y}, x)<0\}$. In the particular case where $F$ is \mbox{pseudomonotone}, using (\ref{eq:equil1}) and taking into account that $B$ is convex (see assumption $(ii^*)$, we conclude that $\bar{y}\in B$. But this contradicts that $F(x,x)=0$ and the affirmation  is proved.  
	
After this discussion, follow the existence result for EP that comes as an application of Theorem~\ref{TheoremKKM}.
\begin{theorem}\label{teo1}
	If $F(\cdot,\cdot)$ is a bifuntion satisfying assumptions (i),($ii^*$),(iii) and (iv), then  $EP$ defined in (\ref{ep}) admits a solution. 
\end{theorem}
\begin{remark}
An example that illustrating the usefulness of the previous result, in the sense that it applies to some situations not covered in the linear configuration can be found for example in \cite{bento2021}. It is worth mentioning that a result of similar existence was presented in \cite{colao2012equilibrium} by considering instead of assumptions (i) and (iv) the following strongest hypotheses:
\begin{enumerate}
    \item[($i^*$)] $F(\cdot,\cdot)$ is monotone, i.e.,  for each $(x,y)\in \Omega\times \Omega$, $F(x,y)+F(y,x)\leq 0$;
    \item[($iv^*$)] there exists a compact set $L\subset M$ and a point $y_0\in \Omega\cap L$ such that $F(x,y_0)<0, \quad x\in\Omega\setminus L$. 
\end{enumerate}
It is not difficult to see that ($i^*$) implies (i), and ($iv^*$) implies (iv).
\end{remark}

\subsection{Resolvents of bifunctions}

In this section, we present a new proposal for a resolvent in the Riemannian context associated with the bifunction $F(\cdot,\cdot)$ given as in (\ref{ep}). As noted in the introduction of the paper, a first definition of resolvent in that setting has appeared in \cite{colao2012equilibrium}. It is the set-valued operator $J_\lambda^F:M\rightrightarrows\Omega$, $\lambda>0$, given as in (\ref{bi-Function}) which, despite being a natural extension of the one introduced in the linear setting in \cite{combettes2005equilibrium}, its well-definedness as well as of the proximal point algorithm for solving EP depend on the convexity of the function $M\ni y\mapsto \langle u_z, \mbox{exp}^{-1}_zy\rangle$, $u_z\in T_zM$ that has been shown not to happens in general; see \cite{wang2016some,cruz2017note}.

Our alternative definition for the resolvent associated with $F (\cdot,\cdot)$ is given as follows: 
\begin{equation}\label{Newbi-Function}
J_{\lambda}^F(x):=\{z\in \Omega : \lambda F(z,y)+d(z,x)b_{\gamma_{z,x}}(y)\geq 0,\quad y\in \Omega\},
\end{equation}
where $\gamma_{z,x}:[0,+\infty[\to M$ is a geodesic ray parametrized by arc length starting from $z$ passing through $x$ and $b_{\gamma_{z,x}}(y)=\lim_{t\to+\infty}[d(y,\gamma_{z,x}(t))-t]$. It is not difficult to see that  $d(z,x)b_{\gamma_{z,x}}(y)=\langle z-x,y-z \rangle$ in the linear setting, showing that this new proposal in (\ref{Newbi-Function}) also retrieves the model proposed and explored in \cite{combettes2005equilibrium}.
Moreover, the new term that plays the role of regularization is now a convex function in general Hadamard manifolds, being a first step to fully answer to the problem posed in \cite[Section 5]{cruz2017note}. 

\begin{theorem}
Let $F(\cdot,\cdot)$ be a bifunction monotone and consider $\lambda>0$. Then, one has that the application $\Omega \ni (z,y)\mapsto F_{\lambda,x}(z,y)=\lambda F(z,y)+d(z,x)b_{\gamma_{z,x}}(y)$ is monotone. Moreover, if the assumptions (ii)-(iv) in Theorem~\ref{teo1} hold, then $J_\lambda^F(x)\neq\emptyset$ for all $x\in M$, $J_\lambda^F(\cdot)$ is single-valued and the fixed point set of $J_\lambda^F(\cdot)$ is the equilibrium point set of $F$.

\end{theorem}
\begin{proof}
Take $x\in M$. For each $(z, y)\in\Omega\times \Omega$, from the definition of the Busemann's function we obtain:
\begin{equation}\label{ineq-Busemann1}
b_{\gamma_{z,x}}(y)\leq d(y,x)-d(z,x),\quad b_{\gamma_{y,x}}(z)\leq d(z,x)-d(y,x). 
\end{equation}
Combining two last inequalities with definition of $F_{\lambda,x}(\cdot,\cdot)$ and using that $F(\cdot,\cdot)$ is monotone, we have:
\[
F_{\lambda,x}(z,y)+F_{\lambda,x}(y,z)\leq -( d(z,x)-d(y,x))^2,
\]
and the first part of the result is proved. Since $F_{\lambda,x}(y,y)=0$ for all $y\in\Omega$, to prove the second part it is sufficient to ensure that $F_{\lambda,x}(\cdot,\cdot)$, also satisfies conditions ($ii^*$), (iii), (iv).  Taking into account that $F(\cdot,\cdot)$ fulfills condition (ii),  $F_{\lambda,x}(\cdot,\cdot)$ satisfies $(ii^*)$ because $b_{\gamma_{z,x}}(\cdot)$ is convex and $F_{\lambda,x}(z,\cdot)$ as a sum of convex functions is also a convex function. To see that $F_{\lambda,x}(\cdot,y)$ also satisfies condition (iii) just note that from Lemma~\ref{lemmaUpperSem} $b_{\gamma_{(\cdot),x}}(y)$ it is upper semicontinuously and, as $F(\cdot,\cdot)$ satisfies (ii), then $F_{\lambda,x}(\cdot,y)$ as sum of upper semicontinuously functions it is also upper semicontinuously. Let us now to prove that $F_{\lambda,x}(\cdot,\cdot)$ satisfies (iv). 
First of all, given $z_{0}\in M$, consider a sequence $\{z^k\}\subset\Omega$ such that $\{d(z^k,z_0)\}$ converges to infinity as $k$ goes to infinity and take $\tilde{z}$ a solution of the problem in (\ref{ep}) which exists due Theorem~\ref{teo1}.
	Since $\{d(z^k,z_0)\}$ converges to infinity as $k$ goes to infinity, in particular, $\{d(z^k,z)\}$ also converges to infinity. Moreover, as $\{z^{k}\}\subset\Omega$, monotonicity of $F(\cdot,\cdot)$ implies that $F(z^{k},\tilde{z})\leq 0$. Hence, using first inequality in (\ref{ineq-Busemann1}), we have
	$F_{\lambda,x}(z^k,\tilde{z})\leq d(z^k,x)(d(\tilde{z},x)-d(z^k,\tilde{z}))$, from where we can conclude that $F_{\lambda,x}(\cdot,\cdot)$ satisfies the condition (iv). Therefore, using Theorem~\ref{teo1} it follows that  $J^F_{\lambda}(x)\neq \emptyset$ for all $x\in M$.

    To prove the last part, given $x\in M$, 
	take $z_1, z_2 \in J_{\lambda}^F(x)$. From the definition of the resolvent $J^F_\lambda(\cdot)$ in (\ref{Newbi-Function}), we have
	\begin{equation}\label{ineq1}
	\lambda F(z_1, z_2) +d(z_1,x)b_{\gamma_{z_1,x}}(z_2)\geq 0,
	\end{equation}
	\begin{equation}\label{ineq2}
	\lambda F(z_2,z_1)+d(z_2,x)b_{\gamma_{z_2,x}}(z_1)\geq 0.
	\end{equation}
	Since $F(\cdot,\cdot)$ is monotone and $\lambda>0$, combining inequalities (\ref{ineq1}) and (\ref{ineq2}), we obtain:
\[
d(z_1,x)b_{\gamma_{z_1,x}}(z_2) + d(z_2,x)b_{\gamma_{z_2,x}}(z_1) \geq 0.
\]
The last inequality combined with  (\ref{ineq-Busemann1}) implies that $d(z_1,x) = d(z_2,x)$ and, hence, $b_{\gamma_{z_1,x}}(z_2) \leq 0$ and  $b_{\gamma_{z_2,x}}(z_1) \leq 0$. But this information, combined again with (\ref{ineq1}) and (\ref{ineq2}), allows us to obtain that $F(z_1, z_2) \geq ~0$ and  $ F(z_2, z_1) \geq 0$. Since $F$ is monotone, two last inequalities imply that $F(z_1, z_2) = F(z_2, z_1) = 0$. Consequently, using again (\ref{ineq1}) and (\ref{ineq2}) we have $b_{\gamma_{z_1,x}}(z_2) = b_{\gamma_{z_2,x}}(z_1) = 0$. Now, remember that the function $[0,+\infty[\ni t\mapsto \psi(t)= d(z_2,\gamma_{z_1,x} (t))-t$ is non-increasing and $b_{\gamma_{z_1,x}}(z_2) = \lim_{t \to + \infty} \psi(t) = 0$. Moreover, as $d(z_1,x) = d(z_2,x)$, in particular we have $\psi(d(z_1,x)) =0$. Thus, it is easy to see that $\psi(\alpha) = 0$ for all $\alpha > d(z_1, x)$. Hence, for $\alpha > d(z_1,x)$ we have:
	\[
	d(z_2, \gamma_{z_1,x} (\alpha)) = \alpha = \alpha - d(x,z_1) + d(x,z_2) = d(x, \gamma_{z_1,x} (\alpha)) + d(x,z_2),
	\]
	from which we can conclude that $z_1=z_2$, which concludes the proof of the theorem. 
	\end{proof}

\end{document}